\documentclass{article}

\usepackage[utf8]{inputenc}
\usepackage[T1]{fontenc}
\usepackage{amsmath}
\usepackage{amsthm}
\usepackage{amssymb}
\usepackage{mathrsfs}
\usepackage{mathtools}
\usepackage{etoolbox}
\usepackage{stmaryrd}
\usepackage{amsfonts}
\usepackage{pgfmath}
\usepackage{modulus}
\usepackage{tikz}
\usepackage{tikz-3dplot}
\usepackage{wrapfig}
\usepackage{subcaption}
\usepackage[export]{adjustbox}
\usepackage{listings}
\usepackage{graphicx}
\usepackage{hyperref}

\newtheorem{introtheorem}{Theorem}

\newtheorem{theorem}{Theorem}[section]
\newtheorem*{theorem*}{Theorem}

\newtheorem{lem}[theorem]{Lemma}
\newtheorem{prop}[theorem]{Proposition}

\newtheorem{defi}[theorem]{Definition}

\theoremstyle{definition}

\theoremstyle{remark}
\newtheorem{rem}[theorem]{Remark}

\newcommand{\intint}[2]{\llbracket #1, #2 \rrbracket}
\newcommand{\N}{\mathbf{N}}
\newcommand{\Z}{\mathbf{Z}}

\newcommand{\defn}{\stackrel{\textrm{\scriptsize def}}{=}}

\renewcommand{\epsilon}{\varepsilon}
\newcommand{\aut}{\text{Aut}}
\newcommand{\ad}{\text{ad}}
\newcommand{\out}{\text{Out}}
\newcommand{\inn}{\text{Inn}}
\newcommand{\inv}[1]{#1^{-1}}
\newcommand{\id}{\text{id}}
\newcommand{\gen}[1]{\left\langle #1 \right\rangle}

\title{The conjugacy problem in $\out(F_m)$ when the polynomial restrictions are non-growing}
\author{Gabriel Bartlett}

\begin{document}

\maketitle

\begin{abstract}
    We prove that the conjugacy problem in $\out(F_m)$ is solvable for the class of outer automorphisms whose restrictions to their polynomial subgroups are of finite order. 
    To do this, we first investigate the structure of suspensions of free groups by automorphisms whose outer class is of finite order.
    We then apply a reduction of our main result to certain problems on groups of this form.
\end{abstract}


\tableofcontents


\section*{Introduction}

    \subsection*{Context}

    In 1911, Dehn introduced three fundamental problems in group theory, including the conjugacy problem.
    A group $G$ has a solvable conjugacy problem if there exists an algorithm which takes two elements of $G$ and responds as to whether or not they are conjugate in $G$.
    An early instance occurs when one classifies conjugacy classes of general linear groups.

    In many classes of groups, explicit solutions to this problem have been found.
    For instance, the problem is solved for free groups simply by considering the corresponding cyclically reduced words, and in general for hyperbolic groups and many groups with hyperbolic geometry.
    In 2006, Bogopolski, Martino, Maslakova, and Ventura solved the problem for free-by-cyclic groups (see \cite{bmmv}).
    
    While the problem is particularly interesting for automorphism groups, it is remarkable that it is not solved for groups of outer automorphisms of free groups.
    However some progress has been made toward a full solution.
    The first observation to make surrounding this particular instance of the conjugacy problem is a strong link with the isomorphism problem among mapping tori.
    Indeed, the mapping torus constructed from an automorphism is isomorphic to another by a fibre- and orientation-preserving isomorphism if and only if the outer automorphisms are conjugate.
    The problem has been solved for finite order outer automorphisms, in \cite{Khram95}.
    In \cite{sela} and in \cite{los}, the problem was solved for irreducible automorphisms.
    It has also been solved for Dehn twists in \cite{Dehntwists},  for automorphisms of linear growth in \cite{klv}, and in \cite{dahm} it was solved for atoroidal automorphisms.
    In \cite{reduc_conj_pb}, Dahmani and Touikan give a reduction of this problem to a number of sub-problems concerning mapping tori.
    The conjugacy problem in $\out(F_m)$ is therefore reduced to the study of groups of the form $F_m \rtimes_{\varphi} \gen{t}$, and their subgroups, for so-called polynomially growing automorphisms $\varphi$.
    In \cite{unip_lin_susp}, Dahmani and Touikan apply this reduction to solve the problem for outer automorphisms whose polynomial restrictions are unipotent and linearly growing.
    In \cite{fh}, Feighn and Handel solve it for unipotent polynomially growing (UPG) automorphisms, using a completely independent strategy.
    Furthermore, in \cite{dfmt}, the authors solve the conjugacy problem for all automorphisms in $\out(F_3)$.

    \subsection*{Main result}

    In this paper, we apply Dahmani and Touikan's reduction (see Theorem \ref{t_reduc}) to a certain class $\mathcal{A}$ of outer automorphisms.
    This class consists of automorphisms whose restrictions to their polynomial subgroups are of finite order.
   
    We wish to prove the following theorem (see Theorem \ref{t_conj_pb}).

    \begin{introtheorem}
    \label{t_A}
        Let $F$ be a finitely generated free group.
        There is an algorithm that decides whether any two outer automorphisms of $F$, whose restrictions to their polynomial subgroups are of finite order, are conjugate in $\out(F)$.
    \end{introtheorem}

    Here, a subgroup is polynomial if its conjugacy class is stable by some power of the automorphism and each element grows polynomially by iterative applications of the automorphism in question (see subsection \ref{ss_mt}).

    The automorphisms in Theorem \ref{t_A} may have a very complicated exponentially growing part.
    For instance, decompose a finitely generated free group $F$ as a free product $F = A \ast B_1 \ast \dots \ast B_r$.
    Choose $\varphi_i$ of finite order on $B_i$, and define an automorphism of $F$ as preserving $A$ and being atoroidal on $A$, and as $\ad_{a_i} \circ \varphi_i$ on $B_i$ for $a_i \in A$, with $\ad_{a_i}$ the inner automorphism conjugating by $a_i$.

    To prove Theorem \ref{t_A}, we will need to study specific problems in $F_k \rtimes_{\psi} \gen{t}$, with $[\psi]$ of finite order.
    Among them is the problem of having \textit{congruences that (effectively) separate torsion}.
    This is the case for a class of groups if there is an algorithm that takes a group $G$ in the class, and outputs a subgroup $H < G$, that is characteristic and of finite index, such that $\out(G) \rightarrow \out(G/H)$ has torsion-free kernel (see Definition \ref{d_mink}).
    We must also prove that these groups have a solvable \textit{fibre- and orientation-preserving mixed Whitehead problem}.
    That is to say, given any two tuples of conjugacy classes of tuples of $F_k \rtimes_{\psi} \gen{t}$, we wish to decide whether there exists a fibre- and orientation-preserving outer automorphism sending one to the other.
    These conditions will need to be treated separately.

    We will therefore devote Section \ref{s_mink} to proving that free-by-cyclic groups with finite order monodromy have congruences that effectively separate torsion. In a short digression, we then prove that such groups are also conjugacy separable, though this is not used to prove Theorem \ref{t_A}.
    We will then prove in Section \ref{s_wh} that the same class of groups has a solvable fibre- and orientation-preserving mixed Whitehead problem.
    Finally, we will use these results to prove in Section \ref{s_conj} that $\mathcal{A}$ satisfies the conditions of Theorem \ref{t_reduc} and therefore that the conjugacy problem in $\out(F_m)$ for outer automorphisms in $\mathcal{A}$ is solvable, which will finish the proof of Theorem \ref{t_A}.

    \subsection*{Acknowledgments}
        I thank my doctoral supervisor, François Dahmani, for his time, guidance and insight throughout the development of this proof.
        I also thank the reviewer for their remarkably thorough and meticulous report, that was invaluable in greatly improving this paper and correcting mistakes.
        This research was supported in part by the International Centre for Theoretical Sciences (ICTS) for the program - Geometry in Groups (code: ICTS/GiG2024/07) and in part by the grant ANR-22-CE40-0004 GoFR.

\section{Preliminaries}
\label{prelim}


    \subsection{Mapping tori and Reduction}
    \label{ss_mt}

    Given a group $G$, let $\mathcal{Z}(G)$ denote its center and $\aut(G)$ its automorphism group.
    For $g_0 \in G$, we will denote by $\ad_{g_0} \in \aut(G)$ the inner automorphism $g \mapsto g_0^{-1} g g_0$ and by $\inn(G)$ the normal subgroup of $\aut(G)$ made up of all inner automorphisms of $G$.
    We will also consider the group of outer automorphisms of $G$, $\out(G) = \aut(G) / \inn(G)$.
    A subgroup $H$ of $G$ is said to be \textit{characteristic} if it is preserved by all automorphisms of $G$.
        
    Let $m \in \N^*$.
    Consider $F_m$, the free group with $m$ generators.
    As in \cite{levitt}, we recall the following.
    Let $\Phi \in \out(F_m)$ and $\varphi \in \Phi$.
    An element $g \in F_m$ is said to grow at most \emph{polynomially}  under $\varphi$, if the lengths of the reduced words $\left(\varphi^n(g)\right)_n$ grow at most polynomially.
    It is said to grow at least \emph{exponentially} if the lengths of the reduced words $\left(\varphi^n(g)\right)_n$ grow at least exponentially.
    The same can be said of the conjugacy class $[g]$ under $\Phi$, using the length of the cyclically reduced words.
    This is independent of the choice of basis.
    It is a well-known fact (see \cite[Theorem 6.2]{levitt}) that an element $g \in F_m$ grows either polynomially or exponentially under $\varphi$.
    We say that $\varphi$ is \emph{polynomially growing} if every $g \in F_m$ grows polynomially under $\varphi$ and the same can be said of $\Phi$ if every conjugacy class in $F_m$ grows polynomially under $\Phi$.
    If $P$ is a subgroup of $F_m$, it is \emph{polynomial} if there exist $k \in \N^*$ and $\alpha \in \Phi^k$ such that $\alpha(P) = P$ and $\alpha_{\mid P}$ is polynomially growing.
    A family of subgroups is said to be \emph{malnormal} if for any two of the subgroups $P$ and $Q$ and for $g \in F_m$, if $P \cap \inv{g} Q g$ is non-trivial, then $P = Q$ and $g \in P$.
    \begin{prop}\cite[Proposition 1.4]{levitt}
    \label{t_levitt}
        Let $\Phi \in \out(F_m)$.
        \begin{enumerate}
            \item For $g \in F_m$, $[g]$ grows polynomially under $\Phi$ if and only if it is contained in a polynomial subgroup.
            \item Every non-trivial polynomial subgroup is contained in a unique maximal one.
            \item For $\varphi \in \Phi$ and for $g \neq 1$ that grows polynomially under $\varphi$, the maximal polynomial subgroup $P(g)$ containing $g$ is the set of all polynomially growing elements under $\varphi$.
            \item Maximal polynomial subgroups have finite rank, there are finitely many conjugacy classes of them, and they are malnormal.
        \end{enumerate}  
    \end{prop}
    
    For $\varphi \in \aut(F_m)$, we define the \textit{mapping torus} of $\varphi$,
    $$ \mathbb{T}_{\varphi} = F_m \rtimes_{\varphi} \gen{t_{\varphi}} = \gen{x_1, \dots, x_m, t_{\varphi} \mid \forall 1 \leq i \leq m, \inv{t_{\varphi}} x_i t_{\varphi} = \varphi(x_i)}$$
    
    We will sometimes omit the $\varphi$ from $t_{\varphi}$ if there is no ambiguity.
    Note that all elements of $F_m \rtimes_{\varphi} \gen{t}$ can be written $t^l f$ with $l \in \Z$ and $f \in F_m$.
    We say that $F_m$ is the \textit{fibre} of $\mathbb{T}_{\varphi}$, the coset $t F_m$ determines the positive orientation and that the automorphism $\varphi$ is the \textit{monodromy}.
    Let $\alpha$ and $\beta$ be automorphisms of $F_m$.
    An isomorphism $\psi: F_m \rtimes_{\alpha} \gen{t_{\alpha}} \rightarrow F_m \rtimes_{\beta} \gen{t_{\beta}}$
    is said to be \emph{fibre-preserving} if $\psi(F_m) = F_m$.
    It is \emph{orientation-preserving} if $\psi(t_{\alpha}) \in t_{\beta} F_m$.
    Let $\aut_{fo}(\mathbb{T}_{\varphi})$ (resp. $\out_{fo}(\mathbb{T}_{\varphi})$) denote the group of automorphisms (resp. outer automorphisms) that preserve the fibre and orientation.
    
    Two outer automorphisms $\Phi_1, \Phi_2 \in \out(F_m)$ are conjugate if and only if, for some $\varphi_1 \in \Phi_1$ and $\varphi_2 \in \Phi_2$, $F_m \rtimes_{\varphi_1} \gen{t_1}$ and $F_m \rtimes_{\varphi_2} \gen{t_2}$ are isomorphic by a fibre- and orientation-preserving isomorphism.
    Therefore the conjugacy problem in $\out(F_m)$ is equivalent to the fibre- and orientation-preserving isomorphism problem in the class of groups of the form $F_m \rtimes_{\varphi} \gen{t}$ with $[\varphi] \in \out(F_m)$.

    We give the following definitions as in \cite{reduc_conj_pb}.
    
    \begin{defi}
    \label{d_mink}
    	Say that a group $G$ has \emph{congruences separating torsion} if it has a characteristic subgroup $G_0$ of finite index such that the quotient $G \rightarrow G/G_0$ induces a morphism $\out(G) \rightarrow \out(G/G_0)$ whose kernel is torsion free.
        A class of groups $\mathcal{H}$ has congruences separating torsion \emph{effectively} if there is an algorithm that takes the presentation of a group $G$ in $\mathcal{H}$ and outputs one such $G_0$.
    \end{defi}
    
    \begin{rem}
        This means that to prove that a group $G$ has congruences separating torsion, we must find a finite index characteristic subgroup $G_0$ such that any non-trivial finite order outer automorphism of $G$ is still non-trivial on $G/G_0$.
    \end{rem}
    
    Note for instance that the class of finite rank free groups has congruences effectively separating torsion.
    This can be seen for instance by taking the abelianization $F_m \rightarrow \Z^m$, then the natural quotient $\Z^m \rightarrow \left( \Z/3\Z \right)^m$, and by noting that the kernel of $\out(F_m) \rightarrow GL_m(\Z)$ is torsion free (see \cite[Proposition 1]{BaumslagTaylor}), and so is the kernel of $GL_m(\Z) \rightarrow GL_m(\Z/3\Z)$ by a classical theorem of Minkowski.
    
    Given a group $G$, the \emph{generation problem} asks whether there exists an algorithm which takes a list of elements of $G$ and decides whether it is a generating set of $G$.
    
    \begin{defi}
    \label{d_hat}
        Let $\mathcal{H}$ be a class of groups.
        $\mathcal{H}$ is \emph{hereditarily algorithmically tractable} if the following holds:
        \begin{itemize}
            \item $\mathcal{H}$ is closed under taking finitely generated subgroups and is effectively coherent: finitely generated subgroups are finitely presentable and a presentation is computable,
            \item the presentations of groups in $\mathcal{H}$ are recursively enumerable,
            \item $\mathcal{H}$ has a uniform solution to the conjugacy problem,
            \item $\mathcal{H}$ has a uniform solution to the generation problem.
        \end{itemize}
    \end{defi}
    
    \begin{defi}
    \label{d_smt}
        Let $\varphi \in \aut(F_m)$ and let $H$ be a self-normalising subgroup of $F_m$ whose conjugacy class is preserved by some power of $\varphi$.
        Let $l > 0$ be the smallest integer such that there exists $a \in F_m$ for which $\varphi^l(H) = \ad_a(H)$.
        The \emph{sub-mapping torus} is the subgroup $\gen{H, t^l \inv{a}}$ of $F_m \rtimes_{\varphi} \gen{t}$.
    \end{defi}
    We may note that the sub-mapping torus is isomorphic to $H \rtimes_{\ad_{\inv{a}} \circ \varphi^l}~\gen{t^l \inv{a}}$.
    Since the subgroup $H$ is self-normalising, it does not depend on the choice of $a$.

    A group is said to be small if it does not contain a non-abelian free subgroup.
    We will now state the reduction theorem that will enable us to prove Theorem \ref{t_A}. We will use a version of the reduction theorem with some hypotheses strengthened.
    
    \begin{theorem}\cite[Theorem A]{reduc_conj_pb}
    \label{t_reduc}
        Let $\mathcal{OA}$ be a class of outer automorphisms of $F_m$.
        Let $\mathcal{P}$ be the collection of sub-mapping tori constructed from the maximal polynomial subgroups under each $\Phi \in \mathcal{OA}$.
        Let $\mathcal{P}'$ be the class of all groups in $\mathcal{P}$ and their finitely generated subgroups.
        Assume that $\mathcal{P}$ and $\mathcal{P}'$ satisfy the following:
        \begin{enumerate}
            \item the groups in $\mathcal{P}$ have a recursively enumerable set of presentations and have uniformly solvable conjugacy problem, $\mathcal{P}'$ is a hereditarily algorithmically tractable class of groups, the small subgroups of groups in $\mathcal{P}'$ are finitely generated, and it is possible to decide if a group in $\mathcal{P}'$ coincides with a group that contains it in $\mathcal{P}$.
            \item the fibre- and orientation-preserving isomorphism problem in $\mathcal{P}$ is solvable
            \item $\mathcal{P}'$ has congruences that effectively separate torsion
            \item the fibre- and orientation-preserving mixed Whitehead problem is solvable in every group of $\mathcal{P}$.
        \end{enumerate}
        Then the conjugacy problem in $\out(F_m)$ for automorphisms in $\mathcal{OA}$ is solvable.
    \end{theorem}
    
    Let us denote by $\mathcal{A}$ the set of outer automorphisms whose polynomially growing parts are of finite order. More specifically, it is the set of outer automorphisms $\Phi \in \out(F_m)$ such that for $\varphi \in \Phi$, and for $P$ a polynomial subgroup for $\varphi$, there exist $k \in \N^*$ and $\gamma_P \in F_m$ such that $\ad_{\gamma_P} \circ \varphi^k$ restricts to the identity on $P$.
    
    \begin{rem}
    \label{r_smt}
        Let $\Phi \in \mathcal{A}$, $\varphi \in \Phi$, and take $P$ a self-normalizing polynomial subgroup for $\varphi$.
        The sub-mapping torus constructed from $P$ can be written $P \rtimes_{\ad_{\inv{a}} \circ \varphi^l} \gen{t^l \inv{a}}$, for certain $a$ and $l$ as in the definition.
        Now $P$ is free, and $[\ad_{\inv{a}} \circ \varphi^l] \in \out(P)$ is of finite order.
        Indeed, $l$ is the smallest integer such that $\varphi^l(P)$ is a conjugate of $P$ in $F_m$, therefore $l \mid k$.
        Now $\left(\ad_{\inv{a}} \circ \varphi^l\right)^{\frac{k}{l}} = \ad_{b} \circ \varphi^k$ with $b = \inv{a} \varphi^l(\inv{a}) \dots \varphi^{\left(\frac{k}{l} - 1\right)l}(\inv{a})$.
        Now $\ad_{\gamma_P} \circ \varphi^k$ restricts to the identity on $P$, so on $P$,
        $$\ad_b \circ \varphi^k = \ad_b \circ \varphi^k \circ \inv{\left(\ad_{\gamma_P} \circ \varphi^k\right)} = \ad_{b \inv{\gamma_P}}$$
        And since $P$ is self-normalizing and is preserved by conjugation by $b \inv{\gamma_P}$ in $F_m$, then $b \inv{\gamma_P} \in P$.
        Therefore, in $\out(P)$, $[\ad_{\inv{a}} \circ \varphi^l]^{\frac{k}{l}} = [\ad_{b \inv{\gamma_P}}] = [\id_P]$.
        Therefore every sub-mapping torus constructed from a $\Phi \in \mathcal{A}$ and a finitely generated self-normalizing polynomial subgroup, is of the form $F_M \rtimes_{\alpha} \gen{t_{\alpha}}$, with $[\alpha] \in \out(F_M)$ of finite order (and $M$ not necessarily equal to $m$).
        In particular, this is the case for maximal polynomial subgroups, since they are finitely generated and self-normalizing, by Theorem \ref{t_levitt}.
    \end{rem}

    
    \subsection{Realization of an automorphism and Center}
    \label{s_realization}
    Let $X$ be a connected graph, $v_0 \in X^{(0)}$ a vertex, and let $\sigma: X \rightarrow X$ be an isometry.
    Then we may construct an outer automorphism of the fundamental group $\pi_1(X, v_0)$ induced by $\sigma$.
    Indeed, the isometry $\sigma$ induces an isomorphism $\sigma_*: \pi_1(X, v_0) \rightarrow \pi_1(X, \sigma(v_0))$.
    We then choose a path $c$ from $v_0$ to $\sigma(v_0)$ in $X$.
    Concatenating paths provides us with an isomorphism
    $$\alpha_c: \begin{cases} \pi_1(X, \sigma(v_0)) &\rightarrow \pi_1(X, v_0) \\
    					x &\mapsto c x \bar{c}
    			\end{cases}	
    $$
    where $\bar{c}$ is the path $c$ taken in the other direction.
    Then, by composition, we get an automorphism $\varphi_c = \alpha_c \circ \sigma_* \in \aut(\pi_1(X, v_0))$.
    In fact, through this process we can construct the entire outer class $\Phi$ of $\varphi_c$ by considering all the possible paths $c$.
    Furthermore, this classical construction yields a group homomorphism
    $$ \Omega_{X,v_0} : \begin{cases}
        Isom(X) &\rightarrow \out(\pi_1(X, v_0)) \\
        \sigma &\mapsto \Phi_{\sigma}
    \end{cases}
    $$
    
    We recall the following realization theorem due to Culler in \cite{culler_realization}, stated previously under a more geometric form in \cite[Satz 3.1(a)]{Zimmermann_1981} by Zimmerman, and also due independently to Khramtsov \cite{Khram85}.
    
    \begin{theorem}[Realization theorem]
    \label{t_nielsen}
        \cite[Theorem 2.1]{culler_realization}
        Let $\Phi \in \out(F_m)$ be of finite order.	
        There exists a connected graph $X$ and an isometry $\sigma: X \rightarrow X$, such that there is an isomorphism $\psi : \pi_1(X, v_0) \xrightarrow{\cong} F_m$, with $v_0 \in X$ some vertex, and such that $\sigma$ induces $\Phi$.
        
        In other words, for $\alpha \in \Omega_{X, v_0}(\sigma)$, $[\psi \circ \alpha \circ \inv{\psi}] = \Phi$.
    \end{theorem}
    
    \begin{rem}
    \label{r_nielsen}
        If $\Phi_1, \Phi_2 \in \out(F_m)$ are conjugate, they are realized by the same isometry.
        Assume $\Phi_1 = \Theta \Phi_2 \inv{\Theta}$ and that $\Phi_2$ is realized as in the theorem.
        Consider $\psi' = \theta \circ \psi : \pi_1(X, v_0) \rightarrow F_m$ with $\theta \in \Theta$.
        Then
        $$\Phi_1 = \Theta [\psi \circ \alpha \circ \inv{\psi}] \inv{\Theta} = [\psi' \circ \alpha \circ \inv{\psi'}]$$
        This however is not enough to test conjugacy for finite order outer automorphisms.
        Indeed, one needs be able to enumerate all possible realizations.
        This can be done using a special case of the equivariant Whitehead algorithm described in \cite{klv}.
        Alternatively, one may test the conjugacy of two finite order outer automorphisms using Khramtsov's result in \cite{Khram95}.
    \end{rem}
    
    We recall the following proposition, well-known to experts. (See, for instance, \cite[Proposition 4.1]{LevGBS}.)
    
    \begin{prop}
    \label{pcenter}
    	Let $\Phi \in \out(F_m)$ be of finite order $k$ and let $\varphi \in \Phi$.	
    	If $F_m \rtimes_{\varphi} \gen{t}$ is not isomorphic to $\Z^2$, then $\mathcal{Z}(F_m \rtimes_{\varphi} \gen{t})$ is infinite cyclic.
        In fact, $\mathcal{Z}(F_m \rtimes_{\varphi} \gen{t}) = \gen{t^k f_0}$, where $f_0$ is the element in $F_m$ such that $\varphi^k = \ad_{\inv{f_0}}$. Furthermore, $t$ and $f_0$ commute.
    \end{prop}
    
    \begin{rem}
    \label{r_center}
        The converse is true. In fact if $\Phi$ is not of finite order, the center is trivial. (See, for instance, \cite[Lemma 3.4]{bridson})
    \end{rem}

\section{Congruences effectively separating torsion}
\label{s_mink}

This section is dedicated to proving the following theorem.

\begin{theorem}
\label{mink}
    The class of all mapping tori of the form $F_m \rtimes_{\varphi} \gen{t}$, with $m \geq 1$ and $\varphi \in \aut(F_m)$ such that $[\varphi] \in \out(F_m)$ is of finite order, has congruences effectively separating torsion. 
\end{theorem}

\begin{rem}
\label{r_case_m1}
    The case of $m = 1$ can be done explicitly.
    Indeed, if $m = 1$ we are dealing either with $\Z \times \Z$, or $\Z \rtimes \Z = \gen{f, t \mid \inv{t} f t = \inv{f}}$.
    In the first case, by a classical theorem of Minkowski, the kernel of $GL_2(\Z) \rightarrow GL_2(\Z/3\Z)$ is torsion-free.
    Now $\out(\Z \rtimes \Z)$ is finite, of order 4, and its non-trivial elements are the classes of
    $\alpha : \begin{cases}
        f &\mapsto f \\
        t &\mapsto \inv{t}
    \end{cases},
    \beta : \begin{cases}
        f &\mapsto f \\
        t &\mapsto tf
    \end{cases}$
    and their composition
    $\alpha \circ\beta : \begin{cases}
        f &\mapsto f \\
        t &\mapsto \inv{t}f
    \end{cases}$.
    Now, $[\alpha], [\beta]$ and $[\alpha \circ \beta]$ are of finite order 2.
    Consider the subgroup $\gen{f^2, t^4}$, which is characteristic and of finite index.
    All three outer automorphisms descend to non-trivial outer automorphisms of the quotient, $\Z/2\Z \times \Z/4\Z$.
    \end{rem}

It remains to prove the theorem for $m \geq 2$.


    \subsection{Congruences of the quotient by the center}

    We will first prove the following proposition, which will later allow us to prove Theorem \ref{mink}.

    \begin{prop}
    \label{l_Q_mink}
        The class of all groups of the form
        $$ Q = F_m \rtimes_{\varphi} \gen{t}/\mathcal{Z}(F_m \rtimes_{\varphi} \gen{t})
        $$
        with $m \geq 2$, and $\varphi \in \aut(F_m)$ such that $[\varphi] \in \out(F_m)$ is of finite order $k \geq 1$, has congruences separating torsion effectively.
    \end{prop}

    We will prove this result in the following two steps.

    \begin{prop}
    \label{p_Qa}
        If $Q$ is as in Proposition \ref{l_Q_mink}, then it is virtually free, has no non-trivial finite normal subgroups and its center is trivial.
    \end{prop}

    Note that the center of any virtually free group that is not virtually $\Z$, is in fact finite.
    Indeed, by a corollary of Stallings' theorem, such a group acts properly discontinuously on a tree.
    And since it is not virtually $\Z$, it has at least two loxodromic lines that must be fixed by the center.
    Hence the center is finite.
    Therefore the triviality of the center is automatic here, though it also follows explicitly from the proof of the proposition.
    
    \begin{prop}
    \label{p_virt_conjclasses}
        Let $V$ be a virtually free group with trivial center.
        
        Then $\out(V)$ has finitely many conjugacy classes of finite order elements.
        
        Furthermore, there is an effective procedure to compute a finite list in $\aut(V)$ containing a representative of each conjugacy class of finite order elements of $\out(V)$.

        Finally, the class of all virtually free groups with trivial center, that also have no non-trivial finite normal subgroups, has congruences separating torsion effectively.
    \end{prop}

    Proposition \ref{l_Q_mink} immediately follows from these.

    \begin{proof}[Proof of Proposition \ref{p_Qa}]
        Recall that if $k \geq 1$ is the order of $[\varphi]$ and $f_0 \in F_m$ verifies $\varphi^k = \ad_{\inv{f_0}}$, then $x = t^k f_0$ is such that $\mathcal{Z}(F_m \rtimes_{\varphi} \gen{t}) = \gen{x}$.
        Note that since $\gen{x} \cap F_m = \{1\}$, then $F_m \hookrightarrow{} Q$, ie. $Q$ contains a copy of $F_m$.
        Furthermore, $F_m$ is of finite index $k$ in $Q$ as $Q/F_m \cong \Z/k\Z$ ($Q = \bigsqcup\limits_{i = 0}^{k-1} t^i \bar{F_m})$.
        Hence, $Q$ is virtually free.
        
        In fact, if $k = 1$, $Q$ is free of rank at least $2$ and the result is clear.
        Now suppose $k > 1$.
        Let us show that $Q$ has no non-trivial finite normal subgroups.
        Assume for contradiction that such a subgroup $K < Q$ exists.
        Note that since $K$ is finite, it has no elements of infinite order, and hence $K \cap F_m = \{1\}$.
        Therefore $K \hookrightarrow Q/F_m \cong \Z/k\Z$. 
        So $K$ is finite cyclic, generated by some element $\tau = \bar{t}^a \bar{f_1}$ of $Q$, with $f_1 \in F_m$ and $a \in \intint{1}{k - 1}$.
        
        Let us prove that $K$ is in fact central.
        Take $y \in F_m \rtimes_{\varphi} \gen{t}$ and consider $\inv{\bar{y}} \tau \bar{y}$.
        In $F_m \rtimes_{\varphi}\gen{t}$, $\inv{y} (t^a f_1) y$ can be written $t^a f$ for some $f \in F_m$ since, conjugates in $F_m \rtimes_{\varphi}\gen{t}$ must have the same exponent of $t$.
        Now by denoting $f' = \inv{f_1} f$, we get $t^a f = (t^a f_1) f'$, so that $\inv{\bar{y}} \tau \bar{y} = \bar{t}^a \bar{f} = \tau \bar{f'}$.
        On the other hand, since $K$ is normal, $\inv{\bar{y}} \tau \bar{y} \in K$, and since $K$ is also cyclic generated by $\tau$, $\inv{\bar{y}} \tau \bar{y} = \tau^l$, for some $l \in \Z$.
        Therefore, $\inv{\bar{y}} \tau \bar{y} = \tau^l = \tau \bar{f'}$.
        Hence $\tau^{l-1} = \bar{f'} \in F_m \cap K$.
        Now $\tau \neq 1$, since we took $K$ to be non-trivial,  so $l = 1$.
        Therefore, for all $\bar{y} \in Q$, $\inv{\bar{y}} \tau \bar{y} = \tau$.
        So $K \subset \mathcal{Z}(Q)$.
        Furthermore, note that $\mathcal{Z}(Q)$ is also finite cyclic since $m \geq 2$ and therefore, as for $K$, $F_m \cap\mathcal{Z}(Q) = \{1\}$.
        We may then assume that $K = \mathcal{Z}(Q)$.
        
        Now since $\tau = \bar{t^a} \bar{f_1}$ is central in $Q = F_m \rtimes_{\varphi}\gen{t}/\gen{t^k f_0}$, for all $f \in F_m$ there exists $w \in \Z$ such that $\inv{(t^a f_1)} f (t^a f_1) = f (t^k f_0)^w \in F_m$.
        Also, for all $f \in F_m$, $\inv{(t^a f_1)} f (t^a f_1) \in F_m$.
        Thus $w = 0$ and $t^a f_1$ commutes with all $F_m \hookrightarrow F_m \rtimes_{\varphi} \gen{t}$.
        Recall that $\Phi \in \out(F_m)$ is of finite order $k>1$.
        Now we have just shown that, for $f \in F_m \hookrightarrow F_m \rtimes_{\varphi} \gen{t}$, $\inv{f_1} ( t^{-a} f t^a) f_1 = f$.
        In other words, in $\aut(F_m)$, $\ad_{f_1} \circ \ad_{t^a} = \id_{F_m}$.
        Therefore $\varphi^a = \ad_{t^a} = \inv{\ad_{f_1}} \in \inn(F_m)$.
        So $k \mid a$, which is impossible, as $1 \leq a \leq k - 1$.
        We have therefore reached a contradiction, having assumed that $Q$ had non-trivial finite normal subgroups.
        Hence $Q$ does not have any non-trivial finite normal subgroups, and in particular its center is trivial.
    \end{proof}

    Let us recall two notions that will be important.
    Firstly, a group $G$ is \emph{conjugacy separable} if every element $g\in G$ is \emph{conjugacy distinguished}, ie. for all $h \in G$ such that h and g are not conjugate in $G$, there is a normal subgroup $N \unlhd G$ of finite index, such that the images $\bar{h}$ and $\bar{g}$ in $G/N$ are still not conjugate.
    We will use Dyer's theorem that virtually free groups are conjugacy separable \cite{Dyer_79}.

    The second concerns pointwise inner automorphisms.
    If $G$ is a group, an automorphism $\psi$ is \emph{pointwise inner} if, for all $g \in G$, $\psi(g)$ is conjugate to $g$.
    (We also say that the outer automorphism $[\psi]$ is pointwise inner.)
    In many cases, such automorphisms are inner.
    This was first observed by Grossman in \cite{grossman}, but we will use the comprehensive result of Minasyan and Osin \cite[Corollary 1.2]{Min_Os} that in a (relatively) hyperbolic group without non-trivial finite normal subgroups, all pointwise inner automorphisms are inner.
    
    The relevance of these two notions is given by the following proposition, which we will use to prove Proposition \ref{p_virt_conjclasses}.
    
    \begin{prop}
    \label{p_conj_sep_mink}
        \cite[Proposition 3.2]{reduc_conj_pb}
        Let $G$ be a finitely presented group such that the following holds:
        \begin{enumerate}
            \item $G$ is conjugacy separable,
            \item $G$ has no non-trivial, pointwise inner, finite-order, outer automorphisms,
            \item we are given a finite list $\{\alpha_1, \dots, \alpha_k\} \subset \aut(G)$ containing a representative of the conjugacy class of every finite-order element of $\out(G)$.
        \end{enumerate}
        Then $G$ has congruences separating torsion effectively.
    \end{prop}

    The previous proposition is stated as in \cite[Proposition 3.2]{reduc_conj_pb}, where it is proven that if a group $G$ verifies the hypotheses in the statement, then one may explicitly construct a finite index, characteristic subgroup of $G$ separating torsion.
    Therefore, the conclusion is in fact that the class of groups verifying these hypotheses has congruences effectively separating torsion.

    Recall that the Dunwoody-Stallings deformation space of a group $G$ is the set of decompositions of $G$ as a graph of groups with finite edge groups and finite or one-ended vertex groups.
    As in \cite{dahm_guir}, we recall that this space can be viewed as a graph whose vertices are equivalence classes of certain essential $G$-trees on which $G$ acts such that all edge stabilizers are finite.
    Furthermore, this graph inherits an action of $\out(G)$ from the action of $\out(G)$ by precomposition on the deformation space.
    If $G$ is hyperbolic, a finite fundamental domain for this action can be computed, as well as  the stabilizers of each of its points, by \cite[Theorem 7.1, Lemma 7.5]{dahm_guir}.
    Let us now prove Proposition \ref{p_virt_conjclasses}.
    
    \begin{proof}[Proof of Proposition \ref{p_virt_conjclasses}]
        We apply a variation on Culler's proof of Theorem \ref{t_nielsen} (\cite[Theorem 2.1]{culler_realization}).
        Let $H$ be a finite subgroup of $\out(V)$.
        This proof relies on constructing an extension $\Gamma$ of $V$ by $H$, in the sense that we wish to obtain the following short exact sequence:
        $$ 1 \rightarrow V \rightarrow \Gamma \rightarrow H \rightarrow 1
        $$
        corresponding to the inclusion homomorphism $H \hookrightarrow \out(V)$.
        For this, take $\Gamma$ to be the full preimage of $H$ in $\aut(V)$.
        Since $V$ has trivial center, it is isomorphic to $\inn(V)$, which injects into $\Gamma$, and $\Gamma/\inn(V) = H$.
        Now $V$ is of finite index in $\Gamma$, and $V$ is virtually free.
        Therefore $\Gamma$ is also virtually free.
        Recall that by Karrass, Pietrowski and Solitar's corollary of Stallings's theorem \cite[Theorem 1]{karrass_gog} there exists $\mathcal{X}$ a finite graph of finite groups, and $v_0$ a vertex, such that $\Gamma = \pi_1(\mathcal{X}, v_0)$.
        Now, since $V$ is a subgroup of $\Gamma$, and $\Gamma$ acts on $T$ the Bass-Serre tree of $\mathcal{X}$, then so does $V$.
        Now $V$ does not necessarily act freely on $T$ since $V$ is only virtually free.
        So $V \backslash T$ is a graph of groups whose fundamental group is isomorphic to $V$, and $V\backslash \Gamma \cong H$ acts on $V \backslash T$.
    
        In other words, we have found a tree in the Stallings-Dunwoody space $\mathcal{D}$ of $V$ on which $H$ acts, that is the Bass-Serre tree of the graph of groups $V \backslash T$.
        Thus $H$ fixes a point in $\mathcal{D}$.
        Now, since $V$ is virtually free, it is hyperbolic.
        Therefore, we may apply the results of \cite{dahm_guir} recalled above and compute $(\tau_i)_{1 \leq i \leq N}$ a fundamental domain in $\mathcal{D}$ for the action of $\out(V)$.
        Therefore a conjugate of $H$ fixes some $\tau_i$.
        Hence, there exists $H'$, conjugate to the subgroup $H$, that is in the stabilizer in $\out(V)$ of $\tau_i$, denoted by $\out_{\tau_i}(V)$.
        
        Now, $\out_{\tau_i}(V)$ is finite.
        Indeed, consider $\Psi \in \out_{\tau_i}(V)$.
        Then $\Psi$ is an outer automorphism of $V$ that preserves $\tau_i$ and therefore acts $V$-equivariantly on the underlying tree $T_i$.
        Therefore $\Psi$ is an automorphism of the graph of groups $V \backslash \tau_i$, in the sense defined in \cite[Definition 2.19]{dahm_guir}.
        Now, by the corollary of Stallings's theorem recalled above, and by the fact that two trees are in the same deformation space if and only if they have the same elliptic subgroups, $V \backslash \tau_i$ is a finite graph of finite groups.
        Furthermore, $\Psi$ is determined by a tuple $(F, (\varphi_e), (\varphi_v), (\gamma_e))$ of
        \begin{itemize}
            \item[-] a graph isomorphism $F$ of the underlying graph of $V \backslash \tau_i$,
            \item[-] for each edge $e$ in $V \backslash \tau_i$, an isomorphism $\varphi_e$ between the edge groups of $e$ and $F(e)$,
            \item[-] for each vertex $v$ in $V \backslash \tau_i$, an isomorphism $\varphi_v$ between the vertex groups of $v$ and $F(v)$,
            \item[-] for each edge $e$ in $V \backslash \tau_i$, an element $\gamma_e$ in the vertex group of $F(t(e))$, with $t(e)$ the terminal end of edge $e$,
        \end{itemize}
        with certain extra conditions (see \cite[Definition 2.19]{dahm_guir}).
        For each of these items, there are only finitely many possibilities, since the graph is finite and every vertex and edge group as well.
        Therefore, there are only finitely many possibilities for $\Psi$.
        Hence $\out_{\tau_i}(V)$ is finite.
        
        We have therefore proven that, if $\Phi \in \out(V)$ is of finite order (here we take the finite subgroup $H$ we considered above to be $\gen{\Phi}$), $\Phi$ has a conjugate in one of the $(\out_{\tau_i}(V))_{i \in \intint{1}{N}}$.
        Since each of these is finite, there are only finitely many conjugacy classes of finite order elements of $\out(V)$.
    
        Furthermore, since there is an effective procedure to compute $(\tau_i)_{i \in \intint{1}{N}}$, we may explicitly construct all the graph of group automorphisms of each of the $(\tau_i)_{i \in \intint{1}{N}}$.
        We may then take a representative in $\aut(V)$ of each of these, which will provide us with a finite list in $\aut(V)$ containing a representative of the conjugacy class of every finite-order element of $\out(V)$.

        Let us now suppose that $V$ has no finite normal subgroups.
        Since $V$ has trivial center, by \cite[Corollary 1.2]{Min_Os} all pointwise inner automorphisms are inner.
        Furthermore, since $V$ is virtually free, by \cite{Dyer_79} it is conjugacy separable.
        Therefore we may apply Proposition \ref{p_conj_sep_mink} and conclude.

    \end{proof}    

    \subsection{Congruences of the entire suspension}
    
    We may now prove Theorem \ref{mink} that states that the class of all mapping tori of the form $F_m \rtimes_{\varphi} \gen{t}$, for $m \geq 1$, $\Phi \in \out(F_m)$ of finite order, and $\varphi \in \Phi$, has congruences separating torsion effectively.
    
    \begin{proof}[Proof of Theorem \ref{mink}]
        Recall, we have already proven the case where $m = 1$ in Remark \ref{r_case_m1}.
        Note that the case where $\Phi$ is trivial has been treated by \cite[Corollary 3.4]{reduc_conj_pb}.
        Let us now consider $m \geq 2$, $\Phi \in \out(F_m)$ of finite order $k > 1$, and $\varphi \in \Phi$.

        Recall that if $f_0 \in F_m$ verifies $\varphi^k = \ad_{\inv{f_0}}$, then $x = t^k f_0$ is such that $\mathcal{Z}(\mathbb{T}_{\varphi}) = \gen{x}$.
        Note that, since the center of $F_m \hookrightarrow \mathbb{T}_{\varphi}$ is trivial, $x \notin F_m$.
        Let $Q = \mathbb{T}_{\varphi} / \gen{x}$.
        The following sequence is exact:
        $$ 1 \rightarrow \gen{x} \rightarrow F_m \rtimes_{\varphi} \gen{t} \rightarrow Q \rightarrow 1
        $$
    
        We will show that there exists a finite quotient of $F_m \rtimes_{\varphi}\gen{t}$ on which any non-trivial, finite order outer automorphism of $F_m \rtimes_{\varphi}\gen{t}$ induces an outer automorphism that is still non-trivial.
        Let $\Psi \in \out(F_m \rtimes_{\varphi}\gen{t})$ be non-trivial  of finite order.
        Note that $\gen{x} = \mathcal{Z}(F \rtimes_{\varphi}\gen{t})$ is characteristic.
        Hence we may consider $\Psi_Q$ the outer automorphism induced on the quotient.

        Three cases arise. 
        \paragraph{First case}
            Suppose that there exists $\psi \in \Psi$ such that the restriction $\psi\mid_{\gen{x}}$ is non-trivial.
            Therefore, $\psi(x) = \inv{x}$.
            Indeed, as $\mathcal{Z}(F_m \rtimes_{\varphi}\gen{t})$ is characteristic, the restriction $\psi\mid_{\gen{x}}$ is an automorphism of $\mathcal{Z}(F_m \rtimes_{\varphi}\gen{t}) \cong \Z$.
            Since there are only two such automorphisms, identity and inversion, the restriction must be the latter. 
            Now consider $H_0 = \gen{F_m, x^3} \unlhd F_m \rtimes_{\varphi}\gen{t}$.
            Then $H_0$ is normal and of finite index $3k$ in $\mathbb{T}_\varphi$.
            Indeed, since $x = t^k f_0$ and $t$ and $f_0$ commute, we get $F_m \rtimes_{\varphi}\gen{t} / H_0 \cong \Z/3k\Z$.
            Denote with a bar $\bar{.}$ the induced elements in $F_m \rtimes_{\varphi}\gen{t} / H_0$.
            Since this quotient is abelian, and $\bar{x} \neq \overline{\psi(x)}$ in $F_m \rtimes_{\varphi}\gen{t} / H_0$, they cannot be conjugates.
            Now, $H_0$ is not necessarily characteristic in $\mathbb{T}_{\varphi}$.
            Therefore, we consider instead $H_1 = \bigcap\limits_{\alpha \in Aut(\mathbb{T}_{\varphi})} \alpha(H_0)$.
            Note that $H_1$ is characteristic.
            Let us show that it is still of finite index, and that it can be explicitly computed.
            First, note that for all $\alpha \in \aut(\mathbb{T}_{\varphi})$, $\alpha(H_0)$ is still of finite index $3k$ in $\mathbb{T}_{\varphi}$.
            Furthermore, subgroups of index $3k$ are exactly the stabilizers of points in actions on a finite set.
            And since $\mathbb{T}_{\varphi}$ is finitely generated, there are only finitely many such subgroups.
            Hence, there are only finitely many subgroups in this intersection, and $H_1$ is still of finite index in $\mathbb{T}_{\varphi}$.
            Since these subgroup appear as stabilizers, we may explicitly compute generators for their intersections, by the Reidemeister-Schreier method for the action on a product of finite sets.
            Furthermore, since $H_1 \subset H_0$, $\psi$ still induces a non-inner automorphism on the quotient $\mathbb{T}_{\varphi}/H_1$.
            Note that this quotient is in fact still abelian, as $H_0$ contains the derived subgroup of $\mathbb{T}_{\varphi}$ which is characteristic, and therefore so does $H_1$.
    
        \paragraph{Second Case}
            Suppose that $\Psi_Q$ is non-trivial.
            By Proposition \ref{l_Q_mink}, $Q$ belongs to a class of groups that has congruences separating torsion effectively, so we may explicitly construct a finite index characteristic subgroup $L < Q$ such that $\Psi_Q$ is still non-trivial on the quotient $Q/L$.
    
        \paragraph{Third and final case}
            Suppose that $\Psi_Q$ is trivial and that for all $\psi \in \Psi$, $\psi\mid_{\gen{x}}$ is trivial.
            Take $\psi \in \Psi$ such that $\psi_Q$ is trivial.
            In other words, $\psi$ is in the kernel of the map $\aut(\mathbb{T}_{\varphi}) \rightarrow \aut(\mathcal{Z}(\mathbb{T}_{\varphi})) \times \aut(\mathbb{T}_{\varphi}/\mathcal{Z}(\mathbb{T}_{\varphi}))$.
            Then for all $y \in F_m \rtimes_{\varphi}\gen{t}$, there exists $r \in \Z$ such that $\psi(y) = y x^r$.
            And since $\psi$ itself is non-trivial, there exists $y \in F_m \rtimes_{\varphi}\gen{t}$ such that the integer $r$ is non-zero.
            However, we can write $y$ as $t^l f$ with $f \in F_m$ and $l \in \Z$.
            Hence, since $x = t^k f_0$ and since $t$ and $f_0$ commute by Proposition \ref{pcenter},
            $$\psi(y) = t^l f x^r = t^l x^r f = t^l (t^k f_0)^r f = t^{l + rk} f_0^r f$$
            Furthermore, for all $n \in \N$, $\psi^n(y) = t^{l + nrk} f_0^{nr} f$, since $\psi\mid_{\gen{x}}$ is trivial.
            Now, for all $n \in \N^*$, $nrk$ is a non-zero integer.
            Hence $\psi^n(y)$ and $y$ cannot be conjugate in $F_m \rtimes_{\varphi}\gen{t}$, since they have different exponents in $t$.
            This however is impossible, since we took $\Psi$ to be of finite order.
            Therefore this case does not arise.
        
        \paragraph{In conclusion}
            Let $L_1$ be the preimage of $L$ in $\mathbb{T}_{\varphi}$ by the quotient homomorphism.
            Then $L_1$ is a finite index, characteristic subgroup of $F_m \rtimes_{\varphi}\gen{t}$ that can be explicitly constructed (we can lift a generating set of $L$ to $\mathbb{T}_{\varphi}$ using the word problem in $Q$, then add the center).
            Set $H = H_1 \cap L_1$, which is therefore also a finite index, characteristic subgroup of $F_m \rtimes_{\varphi}\gen{t}$ that can be explicitly constructed, as recalled for $H_1$ above.
            Then $\Psi$ is non-trivial on $F_m \rtimes_{\varphi}\gen{t} / H$, and since $H$ does not depend on the non-trivial finite order outer automorphism $\Psi$ that we considered, $F_m \rtimes_{\varphi}\gen{t}$ does indeed have congruences separating torsion.
            Since this construction is effective, the corresponding class of groups has congruences effectively separating torsion.
    \end{proof}

\subsection{Digression: Conjugacy separability of non-growing suspensions}
    
    As a side result that will not be used in what follows, we may note the following property on the conjugacy separability of non-growing suspensions. 

    \begin{prop}
        Let $m \geq 1$ and let $[\varphi] \in \out(F_m)$ be of finite order.
        Then $F_m \rtimes_{\varphi} \gen{t}$ is conjugacy separable.
    \end{prop}

    \begin{proof}
        In the case where $F_m \rtimes_{\varphi} \gen{t} \cong \Z^2$, conjugacy classes are trivial, so for any two distinct elements of $F_m \rtimes_{\varphi} \gen{t}$, it suffices to take a large enough finite quotient so that they are sent to distinct elements.
        
        Suppose $F_m \rtimes_{\varphi} \gen{t}$ is not $\Z^2$. Recall that if $k > 1$ is the order of $[\varphi]$ and $f_0 \in F_m$ verifies $\varphi^k = \ad_{\inv{f_0}}$, then $x = t^k f_0$ is such that $\mathcal{Z}(F_m \rtimes_{\varphi} \gen{t}) = \gen{x}$.
        Take two elements of $F_m \rtimes_{\varphi} \gen{t}$ that are non conjugate, say $a = t^p f$ and $b = t^q g$.

        Suppose first that $p = q$.
        Let $Q = F_m \rtimes_{\varphi} \gen{t} / \mathcal{Z}(F_m \rtimes_{\varphi} \gen{t})$ and consider the images $\bar{a}$ and $\bar{b}$ in $Q$.
        If $\bar{a}$ and $\bar{b}$ are conjugate in $Q$, then there is $\bar{y} \in Q$ such that $\bar{b} = \inv{\bar{y}} \bar{a} \bar{y}$.
        Then by lifting the relation to $F_m \rtimes_{\varphi} \gen{t}$, there exists $r \in \Z$ such that $b = \inv{y} a y x^r$.
        Now, since $p = q$, then $r = 0$.
        So $a$ and $b$ are conjugate and we reach a contradiction.
        Therefore, $\bar{a}$ and $\bar{b}$ are not conjugate in $Q$.
        Now, since $Q$ is virtually free, by \cite{Dyer_79} it is conjugacy separable.
        So there is a finite index, normal subgroup $\bar{N} \unlhd Q$ such that $\bar{a}$ and $\bar{b}$ are still not conjugate on the quotient $Q / \bar{N}$.
        Take $N$ to be the full preimage of $\bar{N}$ by the quotient.
        It is a finite index normal subgroup such that $a$ and $b$ are still not conjugate on the quotient $F_m \rtimes_{\varphi} \gen{t}/N$.

        Now suppose that $p \neq q$.
        Take the quotient by $F_m$, thus sending $a$ and $b$ to two different elements of $\gen{t}$.
        Then take a sufficiently large finite quotient so that $a$ and $b$ are still sent to different elements.
        Then $a$ and $b$ are still not conjugate in this finite quotient as it is abelian.
    \end{proof}

\section{Mixed Whitehead problem}
\label{s_wh}

Take $\Phi \in \out(F_m)$ to be of finite order $k$ and let $\varphi \in \Phi$.
We wish to prove that the mixed Whitehead problem on $\out_{fo}(F_m \rtimes_{\varphi} \gen{t})$ is solvable in the sense recalled in Theorem \ref{t_mwh} below.
For this we will need a result of Dahmani and Guirardel on the extended isomorphism problem for hyperbolic groups with marked peripheral structure. As in \cite[Section 2.4]{dahm_guir}, a marked peripheral structure $\mathcal{P}$ on a group $G$ is a tuple $\mathcal{P} = ([S_1], \dots, [S_n])$ where each $S_i$ is a tuple $(g_1, \dots, g_r)$ of elements of $G$, and its conjugacy class $[S_i]$ is the set of tuples $(\inv{g} g_1 g, \dots, \inv{g} g_r g)$ for $g \in G$. 
Denote by $\aut(G, \mathcal{P})$ the automorphisms $\varphi$ of $G$ such that for all $i$, $[\varphi(S_i)] = [S_i]$, in the sense that there exists $g \in G$ such that $(\varphi(g_1), \dots, \varphi(g_r)) = (\inv{g} g_1 g, \dots, \inv{g} g_n g)$.

\begin{theorem}\cite[Theorem 8.1]{dahm_guir}
\label{t_dahm_guir_hyp}
    Let $G$ be a hyperbolic group, $N \in \N$ and let $S_1, \dots, S_N, S_1', \dots, S_N' \subset G$ be tuples.
    There is an effective procedure to determine if there exists $\Psi \in \out(G)$ such that, for $i \in \intint{1}{N}$, $\Psi([S_i]) = [S_i']$.
    
    Furthermore, let $\mathcal{P} = ([S_1], \dots, [S_N])$.
    Then $\aut(G, \mathcal{P})$ is finitely generated and there is an effective procedure that gives a finite generating set.
\end{theorem}

\begin{rem}
    Since inner automorphisms preserve all conjugacy classes, the same statement can be made with $\out(G, \mathcal{P}) \defn \aut(G, \mathcal{P})/\inn(G)$.
\end{rem}

\begin{rem}
\label{r_exist_construct}
    If there exists such a $\Psi$, sending one tuple $\mathcal{P}$ to another $\mathcal{P'}$, then there is also an effective procedure that produces one.
    Indeed, since $\out(G)$ is finitely generated and there is an effective procedure that computes a finite generating set (see \cite[Theorem 7.1]{dahm_guir}), then we may simply go through all the words of $\out(G)$ of increasing length, testing at each step if every $[S_i]$ is sent to $[S_i']$.
    This procedure terminates, since we assumed that such a $\Psi$ exists, and it returns a suitable outer automorphism.
\end{rem}

Let us now state and prove the main theorem of this section.

\begin{theorem}
\label{t_mwh}
    Let $m \geq 1$, let $\Phi \in \out(F_m)$ be of finite order and let $\varphi \in \Phi$.
    Given two tuples of conjugacy classes of tuples in $F_m \rtimes_{\varphi} \gen{t}$, there is an effective procedure to determine if there exists $\Psi \in \out_{fo}(F_m \rtimes_{\varphi} \gen{t})$ sending one to the other.
\end{theorem}

\begin{proof}
    First let us note that in the case of $m = 1$, if we are dealing with $\Z \times \Z$, then the problem is solved by elementary linear algebra on $\Z$, and if we are dealing with $\Z \rtimes \Z$ it is trivial since, as stated above in Remark \ref{r_case_m1}, $\out(\Z \rtimes \Z)$ is finite.

    Similarly, one may treat the case $m = 2$ explicitly since only a few subcases arise.
    However the proof is significantly more technical than for $m = 1$ and the following covers this case.

    Now let $m \geq 2$.
    Let $\mathcal{P} = ([S_1], \dots, [S_N])$ and $\mathcal{P'} = ([S'_1], \dots, [S'_N])$ be tuples of classes of tuples in $F_m \rtimes_{\varphi} \gen{t}$.
    We begin by checking that for all $i$, the tuples $S_i$ and $S_i'$ have the same length and the exponents in $t$ of each element in $S_i$ matches with the corresponding one in $S_i'$.
    
    As before, we will consider the problem in $Q = F_m \rtimes_{\varphi} \gen{t} / \mathcal{Z}(F_m \rtimes_{\varphi} \gen{t})$, then lift the result to $F_m \rtimes_{\varphi} \gen{t}$.
    Recall that if $k$ is the order of $\Phi$ and $\varphi^k = \ad_{\inv{f_0}}$, then $\mathcal{Z}(F_m \rtimes_{\varphi} \gen{t}) = \gen{t^kf_0}$.
    Recall also that $Q$ is virtually free, and therefore hyperbolic.
    Let $\bar{\mathcal{P}}$ and $\bar{\mathcal{P'}}$ denote the projected tuples in $Q$.
    
    Note that for $\psi \in \aut_{fo}(F_m \rtimes_{\varphi} \gen{t})$, $\psi(t^kf_0)$ is central, so $\psi(t^kf_0) \in \gen{t^kf_0}$.
    However $\psi(t^kf_0) \in t^kF_m$, so $\psi(t^kf_0) = t^kf_0$.
    Furthermore, the natural homomorphism from $\aut_{fo}(F_m \rtimes_{\varphi} \gen{t})$ to $\aut(Q)$ induces a homomorphism $\pi$ between $\aut_{fo}(F_m \rtimes_{\varphi} \gen{t})$ and automorphisms of $Q$ that preserve $F_m$, and send $\bar{t}$ into $\bar{t}F_m$.
    Let us prove that $\pi$ is in fact a bijection.
    Let $\psi \in \aut_{fo}(F_m \rtimes_{\varphi} \gen{t})$ be such that its image is the identity on $Q$.
    Then for all $t^lf \in F_m \rtimes_{\varphi} \gen{t}$, $\psi(t^lf) = t^lf (t^kf_0)^j$, for some $j \in \Z$.
    Since the exponents in $t$ of $\psi(t^lf)$ and $t^lf$ must be equal, $j=0$.
    Hence $\psi$ is the identity on $F_m \rtimes_{\varphi} \gen{t}$ and $\pi$ is injective.
    Let us now show that $\pi$ is surjective.
    Let $x_1, \dots, x_m$ be generators of $F_m$.
    Let $\omega \in \aut(Q)$ send $F_m$ to $F_m$, and $\bar{t}$ to $\bar{t} q$, for some $q \in F_m$.
    For each $i$, let $g_i = \omega(x_i) \in F_m$.
    Define $\psi$ to be the automorphism on $F_m \rtimes_{\varphi} \gen{t}$, such that for each $i$, $\psi(x_i) = g_i$ and $\psi(t) = t q$.
    Let us show that this does indeed define an automorphism of $F_m \rtimes_{\varphi} \gen{t}$, namely that for all $i$, $\psi(\varphi(x_i))$ and $\inv{\psi(t)} \psi(x_i) \psi(t)$ agree.
    These two elements agree on $Q$, so for some $l \in \Z$, $\psi(\varphi(x_i)) = \inv{\psi(t)} \psi(x_i) \psi(t) (t^kf_0)^l$.
    Since they are both elements of $F_m$, then $l = 0$, and they agree on $F_m \rtimes_{\varphi} \gen{t}$.
    So $\psi$ defines an automorphism in $\aut_{fo}(F_m \rtimes_{\varphi} \gen{t})$ that is sent to $\omega$.
    Hence $\pi$ is a bijection.
    
    Now let us check that this lift behaves well with sending $\mathcal{P}$ to $\mathcal{P}'$.
    In other words, if $\psi \in \aut_{fo}(F_m \rtimes_{\varphi} \gen{t})$ is such that $\pi(\psi)$ sends $\bar{\mathcal{P}}$ to $\bar{\mathcal{P}'}$, let us prove that $\psi$ sends $\mathcal{P}$ to $\mathcal{P}'$.
    Let $i \in \intint{1}{N}$.
    There exists $t^{n_i}g_i \in F_m \rtimes_{\varphi} \gen{t}$, such that for all $t^n f \in S_i$ and $t^n f' \in S_i'$, at the same position, we get in $Q$, $\pi(\psi)(\bar{t}^n f) = \inv{(\bar{t}^{n_i}g_i)} \bar{t}^n f' (\bar{t}^{n_i}g_i)$.
    So, for some $l \in \Z$, we get in $F_m \rtimes_{\varphi} \gen{t}$, $\psi(t^n f) = \inv{(t^{n_i}g_i)} t^n f' (t^{n_i}g_i) (t^k f_0)^l$.
    And since $\psi$ is fibre- and orientation-preserving, $l = 0$.
    Hence $\psi$ sends $\mathcal{P}$ to $\mathcal{P}'$.    
    Therefore, the problem is equivalent to deciding whether there exists $\bar{\psi} \in \aut(Q)$, sending $\bar{\mathcal{P}}$ to $\bar{\mathcal{P'}}$, that preserves $F_m$, and sends $\bar{t}$ into $\bar{t}F_m$.

    Since $Q$ is hyperbolic, by Theorem \ref{t_dahm_guir_hyp}, there is an effective procedure to determine if there exists $\bar{\psi} \in \aut(Q)$ sending $\bar{\mathcal{P}}$ to $\bar{\mathcal{P'}}$.
    And by Remark \ref{r_exist_construct}, if such a $\bar{\psi}$ exists, it can be constructed effectively.
    This however does not entirely solve the problem.
    Indeed, $\bar{\psi}$ does not necessarily send $F_m$ to $F_m$ nor $\bar{t}$ into $\bar{t} F_m$.
    If it does not (which can be tested in finitely many steps, by looking at the images of the generators $x_1, \dots, x_m$ and $\bar{t}$), we modify $\bar{\psi}$.
    For this, we use the second part of Theorem \ref{t_dahm_guir_hyp}.
    There is an effective procedure that gives a finite generating set of $\aut(Q, \mathcal{\bar{P}'})$, the subgroup of $\aut(Q)$ that preserves the conjugacy classes in $\mathcal{\bar{P}'}$.
    Therefore we may compose $\bar{\psi}$ by these automorphisms without undoing our work.
    
    Now $\aut(Q)$ acts on the set of surjective maps $Q \rightarrow \Z/k\Z$ by permutation.
    This set is finite, and the kernels of these maps are the index $k$ subgroups of $Q$ that $F_m$ might be sent to by $\bar{\psi}$.
    Now the surjective map that we have has kernel $\bar{\psi}(F_m)$ and sends $\bar{t}$ to some $\bar{p} \in \Z/k\Z$.
    The surjective map that we want has kernel $F_m$ and sends $\bar{t}$ to $\bar{1}$.
    We may effectively compute the subgroup of permutations to which $\aut(Q, \bar{\mathcal{P}'})$ maps by looking at its generators, and decide whether there exists an automorphism $\alpha \in \aut(Q, \bar{\mathcal{P}'})$ sending the map that we have, to the map that we want.
    In which case $\alpha \circ \bar{\psi} \in \aut(Q)$ sends $\bar{\mathcal{P}}$ to $\bar{\mathcal{P}'}$, $F_m$ to $F_m$, and $\bar{t}$ into $\bar{t} F_m$.   
    This automorphism can then be lifted to an automorphism of $F_m \rtimes_{\varphi} \gen{t}$ that solves the problem.
\end{proof}


\section{Conjugacy problem}
\label{s_conj}

In this section we wish to prove Theorem \ref{t_conj_pb}, which states that the conjugacy problem in $\out(F_m)$, for outer automorphisms whose polynomial parts are of finite order, is solvable.
For this we will use the reduction introduced by Dahmani and Touikan recalled in Theorem \ref{t_reduc}.
We will also need the following lemma on the structure of finitely generated subgroups of mapping tori with finite order monodromy.

\begin{lem}
\label{l_sg}
    Let $M \geq 1$ and let $[\varphi] \in \out(F_M)$ be of finite order.
    Then all finitely generated subgroups of $F_M \rtimes_{\varphi} \gen{t}$ are either free of finite rank, or of the form $F_{M'} \rtimes_{\alpha} \gen{t_{\alpha}}$, for some $M'$ finite, and with $[\alpha] \in \out(F_{M'})$ of finite order.
\end{lem}
\begin{proof}
    Let $H$ be a finitely generated subgroup of $F_M \rtimes_{\varphi} \gen{t}$.
    Let $C$ denote the center of $F_M \rtimes_{\varphi} \gen{t}$, and recall that $C$ is infinite cyclic, unless $F_M \rtimes_{\varphi} \gen{t} \cong \Z^2$, by Proposition \ref{pcenter}.
    In the latter case, the subgroups are of the desired form.
    Let us therefore consider the former case.
    
    Let us first suppose that $ H \cap C = \{1\}$.
    Then $H \hookrightarrow (F_M \rtimes_{\varphi} \gen{t})/C$, which is virtually free, as we have seen previously (Proposition \ref{l_Q_mink}).
    Now $H$ is torsion free, as it is a subgroup of $F_M \rtimes_{\varphi} \gen{t}$, which is torsion free.
    Therefore $H$ is in fact free and finitely generated.

    Let us now suppose that $C_0 \defn H \cap C$ is non-trivial.
    Then $C_0$ is infinite cyclic.
    Let $x \in C_0$ be a generator.
    Consider the following short exact sequence.
    $$ 1 \rightarrow F_M \rightarrow F_M \rtimes_{\varphi} \gen{t} \rightarrow \Z \rightarrow 1
    $$
    And now consider its restriction to $H$.
    $$ 1 \rightarrow F_M \cap H \rightarrow H \rightarrow \Z \rightarrow 1
    $$
    Since the term on the right is free, there is a section from $\Z$ to $H$ and therefore $H$ is of the form $H = (F_M \cap H) \rtimes_{\alpha} \Z$.
    Denote $H_F \defn F_M \cap H$ and let $t_1$ be a generator of the cyclic part of $H$, so that $H = H_F \rtimes_{\alpha} \gen{t_1}$.
    We will now check that $H_F$ (which is free) is finitely generated.
    Consider the quotient $\bar{H}$ of $H$ by $C_0 = \gen{x}$.
    
    Now either $H_F$ intersects $C_0$ non-trivially or they intersect trivially.
    In the first case, the center of $H_F$ is non-trivial, so $H_F \cong \Z$.
    However, if $H \cong \Z \rtimes \Z$, its center would be $\gen{t_1^2}$ which does not intersect $H_F$.
    Hence $H$ is in fact $\Z^2$ and we find what we wanted.
    
    In the second case, $H_F$ is isomorphic to its image $\bar{H_F}$ in the quotient $H/C_0$, and $x \notin H_F$.
    So we can write $x = t_1^r h$, with $r \in \N^*$ and $h \in H_F$.
    Therefore, in the quotient, $\bar{t_1}^r \in \bar{H_F}$, so $\bar{H_F}$ is of finite index in $\bar{H}$, which is finitely generated.
    Hence $\bar{H_F}$ is finitely generated, and so is $H_F$.    
    Therefore, we have shown that $H \cong F_{M'} \rtimes_{\alpha} \gen{t_1}$ and $H$ has a non-trivial center since $C_0 \subset \mathcal{Z}(H)$.
    So $[\alpha]$ is of finite order in $\out(F_{M'})$, by Remark \ref{r_center}.
\end{proof}

\begin{rem}
    This result can be proven in various other ways.
    For instance, the fact that $H_F$ is finitely generated can be reached using \cite[Lemma 3.5]{mut}, which relies on \cite[Proposition 2.3]{fhmappingtori}, or by using Cashen-Levitt's results on the BNS invariant of a polynomial free-by-cyclic group in \cite{cashen_levitt}.
\end{rem}

We may now state and prove the following theorem.

\begin{theorem}
\label{t_conj_pb}
    Let $m \in \N\setminus\{0\}$ and let $F_m$ be the free group of rank $m$.
    Let $\mathcal{A}$ be the set of outer automorphisms $\Phi \in \out(F_m)$ such that for $\varphi \in \Phi$, and for $P$ a polynomial subgroup for $\varphi$, there exist $k \in \N^*$ and $\gamma_P \in F_m$ such that $\ad_{\gamma_P} \circ \varphi^k$ restricts to the identity on $P$ (ie. the polynomially growing parts of $\Phi$ are of finite order).
    
    There is an effective procedure that decides whether any two outer automorphisms in $\mathcal{A}$ are conjugate in $\out(F_m)$.
\end{theorem}

\begin{proof}
    We will prove that the class $\mathcal{A}$ verifies the hypotheses of Theorem \ref{t_reduc}.
    
    Let $\mathcal{P}$ be the collection of sub-mapping tori constructed from the maximal polynomial subgroups under each $\Phi \in \mathcal{A}$.
    Then by Remark \ref{r_smt}, all groups in $\mathcal{P}$ are of the form $F_M \rtimes_{\alpha} \gen{t}$, for some $M$ and some $[\alpha] \in \out(F_M)$ of finite order.
    Note that for every $[\varphi] \in \mathcal{A}$, it is possible to compute a finite presentation for each of its maximal polynomial subgroups.
    To prove this, we will use \cite[Theorem 3]{dahm_guir_presenting}.
    The group $F_m \rtimes_{\varphi} \gen{t_{\varphi}}$ is indeed hyperbolic relative to its maximal polynomial subgroups, it has a solvable word problem since it is residually finite, and the class of all groups of the form $F_{m'} \rtimes_{\psi} \gen{t_{\psi}}$, with $m \geq 1$ and $[\psi]$ of finite order, is recursive.
    To see this last point, for each $m \geq 1$, we consider all finite graphs with $F_m$ as a fundamental group, and consider all of their isometries.
    This yields all the finite order outer automorphisms of $F_m$, by the realization results presented in Section \ref{s_realization}.
    
    First, let us note that, by Theorem \ref{t_mwh}, the fibre- and orientation-preserving mixed Whitehead problem is solvable in every group of $\mathcal{P}$.
    Therefore hypothesis 2 is verified.

    Let us check that the fibre- and orientation-preserving isomorphism problem in $\mathcal{P}$ is solvable.
    Let $\mathbb{T}_1, \mathbb{T}_2 \in \mathcal{P}$.
    If there exists a fibre- and orientation-preserving isomorphism between $\mathbb{T}_1$ and $\mathbb{T}_2$, then the respective fibres are isomorphic, and therefore have the same rank, say $M$.
    Hence the underlying automorphisms, $\alpha_1$ and $\alpha_2$, both belong to $\aut(F_M)$.
    Therefore, a fibre- and orientation-preserving isomorphism exists if and only if $[\alpha_1]$ and $[\alpha_2]$ are conjugate in $\out(F_M)$.
    Now, $[\alpha_1]$ and $[\alpha_2]$ are of finite order in $\out(F_M)$, since they arise from elements of $\mathcal{A}$.
    Therefore, their conjugacy can be tested using \cite[Theorem 3]{Khram95}.
    Hence hypothesis 4 is verified.

    Let $\mathcal{P}'$ be the class of all groups in $\mathcal{P}$ and their finitely generated subgroups. By Lemma \ref{l_sg} all groups in $\mathcal{P}'$ are either free of finite rank, or of the form $F_{M'} 
    \rtimes_{\alpha} \gen{t_{\alpha}}$, for some $M'$ finite, and with $[\alpha] \in \out(F_{M'})$ of finite order.
    
    Let us prove that $\mathcal{P}$ and $\mathcal{P}'$ verify hypothesis 1.
    The groups in $\mathcal{P}$ have a recursively enumerable set of presentations using Tietze transformations (see, for instance, \cite[Section 1.5]{magnus}).
    The class $\mathcal{P}$ has a uniformly solvable conjugacy problem, by \cite[Theorem 1.1]{bmmv}.
    Furthermore, all small subgroups of groups in $\mathcal{P}'$ are finitely generated.
    Indeed, let $H \in \mathcal{P}'$ and let $L < H$ be a small subgroup.
    Either $H$ is free, in which case so is $L$ and therefore $L$ is either trivial of infinite cyclic, or $H$ is of the form $F_{M'} \rtimes_{\alpha} \gen{t}$ with $[\alpha]$ of finite order.
    In this case, as in the proof of Lemma \ref{l_sg}, $L$ can be written $(F_{M'} \cap L) \rtimes \Z$.
    Now $F_{M'} \cap L$ is a free subgroup of $L$, so it must be abelian.
    Therefore, $L$ is either trivial, infinite cyclic, $\Z^2$, or $\Z \rtimes \Z$, so it is indeed finitely generated.

    Let us now prove that $\mathcal{P}'$ is hereditarily algorithmically tractable.
    Recall that, to do this, we must prove the following (see Definition \ref{d_hat}):
    \begin{itemize}
            \item $\mathcal{P}'$ is closed under taking finitely generated subgroups and is effectively coherent: finitely generated subgroups are finitely presentable and a presentation is computable,
            \item the presentations of groups in $\mathcal{P}'$ are recursively enumerable,
            \item $\mathcal{P}'$ has a uniform solution to the conjugacy problem,
            \item $\mathcal{P}'$ has a uniform solution to the generation problem.
        \end{itemize}
    The class $\mathcal{P}'$ is closed under taking finitely generated subgroups.
    Let $H_0$ be a group in $\mathcal{P}'$, given by a presentation.
    Let us prove that it is effectively coherent.
    Let $H = \gen{h_1, \dots, h_r}$ be a subgroup of $H_0$.
    By Lemma \ref{l_sg}, $H_0$ is either free of finite rank, or of the form $F_{M'} \rtimes_{\alpha} \gen{t}$ with $[\alpha]$ of finite order.
    By cycling through its presentations, again using Tietze transformations, we can determine which of these two mutually exclusive situations we are in.
    If $H_0$ is free, then we may use Stallings foldings to determine a basis, and therefore a presentation, of $H$.
    Suppose that $H_0$ is of the form $F_{M'} \rtimes_{\alpha} \gen{t}$ with $[\alpha]$ of finite order.
    Then by \cite[Proposition 4.1]{LevGBS}, $H_0$ admits a presentation as the fundamental group of a graph of cyclic groups.
    We may compute this presentation by enumerating the presentations of $H_0$.
    We will now use a criterion of Kapovich-Myasnikov-Weidmann for effective coherence that uses the notion of a benign graph of groups (see \cite[Definition 5.6]{kmw}).
    In our case, since the vertex and edge groups are all infinite cyclic, the graph of groups in question is benign.
    Therefore, we may apply \cite[Theorem 5.8.b]{kmw} that provides us with a finite presentation of $H$.
    Hence $H_0$ is effectively coherent.

    Furthermore, the presentations of groups in $\mathcal{P}'$ are again recursively enumerable using Tietze transformations.
    There is a uniform solution to the conjugacy problem for groups in $\mathcal{P}'$.
    To do this, we must first check whether the subgroup we are handling is free or a free-by-cyclic group with finite order monodromy.
    Since these are the only two possibilities for finitely generated subgroups of groups in $\mathcal{P}$, by Lemma \ref{l_sg}, it suffices to cycle through the presentations of the subgroup until we find one of these two forms.
    If it is free, we consider the cyclically reduced words, and if not, we may use \cite[Theorem 1.1]{bmmv}.
    
    Finally let us show that the generation problem is solvable.
    Let $H = \gen{h_1, \dots, h_r}$ be a group in $\mathcal{P}'$, and let $H' = \gen{h_1', \dots, h_s'}$ be a subgroup of $H$.
    Let $F_M \rtimes_{\varphi} \gen{t}$ be the group in $\mathcal{P}$ containing $H$.
    Then, again using the notion of a benign graph of groups, we may apply \cite[Theorem 5.13]{kmw} to deduce that $F_M \rtimes_{\varphi} \gen{t}$ has a solvable membership problem.
    Therefore, we may check that each $h_i$ is in $H'$.
    The class $\mathcal{P}'$ is therefore hereditarily algorithmically tractable.

    It is also possible to determine if a given group $H'$ in $\mathcal{P}'$ coincides with a group $H$ containing it in $\mathcal{P}$, as this is simply an instance of the generation problem, which we have just addressed.
    Therefore, hypothesis 1 is verified.

    Finally let us show that $\mathcal{P}'$ has congruences that effectively separate torsion.
    Let $H$ be a group in $\mathcal{P}'$, and let $H_0$ be a group containing it in $\mathcal{P}$.
    A finite presentation of $H$ can be computed, as has been noted at the start of the proof, and $H$ is effectively coherent, since it is also in $\mathcal{P}'$.
    Therefore, we may explicitly compute a finite presentation of $H$.
    By cycling through its presentations, again using Tietze transformations, we may determine whether $H$ is free of finite rank or free-by-cyclic with finite order monodromy.
    In the first case, we may apply \cite[Proposition 1]{BaumslagTaylor} and a classical theorem of Minkowski as recalled above in Section \ref{ss_mt}.
    In the latter case, we may apply Theorem \ref{mink}.
    Therefore, hypothesis 3 is verified.
    
    We may therefore apply Theorem \ref{t_reduc} to $\mathcal{A}$ and conclude.
\end{proof}

Note that the arguments in this paper rely on the fact that the mapping tori we consider have non-trivial center, and that the quotient by the center is virtually free.
If the base group is hyperbolic instead of free, it is tempting to try to adapt the argument, however it requires some care, for instance in the steps using Dyer's theorem.

\bibliographystyle{alpha}
\bibliography{Biblio.bib}

\textsc{Gabriel Bartlett, Université Grenoble-Alpes, Institut Fourier, 38610 Gières, France}

email. \textbf{gabriel.bartlett@univ-grenoble-alpes.fr}

\end{document}